\documentclass[reqno]{amsart}
\usepackage{aliascnt}
\usepackage{hyperref}
\usepackage{amsfonts}
\usepackage{mathrsfs}
\usepackage{amsmath}
\usepackage{tikz-cd}
\usepackage{mathtools}
\usepackage{amsmath,graphics}
\usepackage{amssymb}
\usepackage{indentfirst,latexsym,bm,amsmath,pstricks,amssymb,
amsthm,graphicx,fancyhdr,float,color}
\usepackage[all]{xy}
\usepackage{amsthm} 
\usepackage{amscd}
\usepackage{hyperref}
\usepackage[all]{hypcap}
\usepackage{multirow}
\usepackage[normalem]{ulem}
\usepackage{braket}
\usepackage{quiver}

\hypersetup{backref,
colorlinks=true}

  \renewenvironment{thebibliography}[1]{
    \begin{oldthebibliography}{#1}
      \setlength{\parskip}{0ex}
      \setlength{\itemsep}{0ex}
  }
  {
    \end{oldthebibliography}
  }

\begin{document}
\newtheorem{Thm}{Theorem}   
\def\thetheorem{\unskip}
\def\theproposition{\unskip}
\def\theconjecture{\unskip}
\newtheorem{Cor}[Thm]{Corollary}
\newtheorem{Lem}[Thm]{Lemma}
\def\thelemma{\unskip}
\theoremstyle{definition}
\newtheorem{Defn}{Definition}
\def\thedefinition{\unskip}
\newtheorem{Notation}[Defn]{Notation}
\newtheorem{remark}[Defn]{Remark}
\def\theremark{\unskip}
\def\thequestion{\unskip}
\newtheorem{Exam}[Defn]{Example}
\def\theexample{\unskip}
\numberwithin{Thm}{section} \numberwithin{Defn}{section}
\newtheoremstyle{case}{}{}{}{}{}{.}{ }{}
\theoremstyle{case}
\newtheorem{case}{Case}

\def\hal{\unskip\nobreak\hfil\penalty50\hskip10pt\hbox{}\nobreak
\hfill\vrule height 5pt width 6pt depth 1pt\par\vskip 2mm}

\renewcommand{\labelenumi}{(\roman{enumi})}
\newcommand{\A}{\mathfrak{A}}
\newcommand{\mfb}{\mathfrak{b}}
\newcommand{\mfC}{\mathfrak{C}}
\newcommand{\mfD}{\mathfrak{D}}
\newcommand{\mfe}{\mathfrak{e}}
\newcommand{\mff}{\mathfrak{f}}
\newcommand{\mfg}{\mathfrak{g}}
\newcommand{\gl}{\mathfrak {gl}}
\newcommand{\mfh}{\mathfrak{h}}
\newcommand{\mfl}{\mathfrak{l}}
\newcommand{\mfm}{\mathfrak{m}}
\newcommand{\mfp}{\mathfrak{p}}
\newcommand{\mfS}{\mathfrak{S}}
\newcommand{\so}{\mathfrak{so}}
\newcommand{\ssp}{\mathfrak{sp}}
\newcommand{\mfsl}{\mathfrak{sl}}
\newcommand{\mft}{\mathfrak{t}}
\newcommand{\mfX}{\mathfrak{X}}
\newcommand{\mfz}{\mathfrak{z}}
\newcommand{\mfu}{\mathfrak{u}}
\newcommand{\mbA}{\mathbb{A}}
\newcommand{\mbE}{\mathbb{E}}
\newcommand{\mFp}{\mathbb{F}_{p}}
\newcommand{\mFq}{\mathbb{F}_{q}}
\newcommand{\mbG}{\mathbb{G}}
\newcommand{\mbP}{\mathbb{P}}
\newcommand{\mbR}{\mathbb{R}}
\newcommand{\mbV}{\mathbb{V}}
\newcommand{\mbZ}{\mathbb{Z}}
\newcommand{\mbN}{\mathbb{N}}

\newcommand{\mcA}{\mathcal{A}}
\newcommand{\mcB}{\mathcal{B}}
\newcommand{\mcC}{\mathcal{C}}
\newcommand{\mcE}{\mathcal{E}}
\newcommand{\mcF}{\mathcal{F}}
\newcommand{\mcG}{\mathcal{G}}
\newcommand{\mcH}{\mathcal{H}}
\newcommand{\mcK}{\mathcal{K}}
\newcommand{\mcM}{\mathcal{M}}
\newcommand{\mcO}{\mathcal{O}}
\newcommand{\mcP}{\mathcal{P}}
\newcommand{\mcQ}{\mathcal{Q}}
\newcommand{\mcR}{\mathcal{R}}
\newcommand{\mcS}{\mathcal{S}}
\newcommand{\mcU}{\mathcal{U}}
\newcommand{\mcZ}{\mathcal{Z}}
\newcommand{\msC}{\mathscr{C}}
\newcommand{\msE}{\mathscr{E}}
\newcommand{\msH}{\mathscr{H}}
\newcommand{\mcJ}{\mathscr{J}}
\newcommand{\msN}{\mathscr{N}}
\newcommand{\msR}{\mathscr{R}}
\newcommand{\msS}{\mathscr{S}}
\newcommand{\msW}{\mathscr{W}}
\newcommand{\B}{\mathrm{B}}
\newcommand{\C}{\mathrm{C}}
\newcommand{\cx}{\mathrm{cx}}
\newcommand{\mrcom}{\mathrm{Com}}
\newcommand{\D}{\mathrm{D}}
\newcommand{\mrdim}{\mathrm{dim}}
\newcommand{\srk}{\mathrm{srk}}
\newcommand{\rank}{\mathrm{rank}}
\newcommand{\mrlie}{\mathrm{Lie}}
\newcommand{\mrlt}{\mathrm{LT}}
\newcommand{\mmax}{\mathrm{max}}
\newcommand{\mrmax}{\mathrm{Max}}
\newcommand{\mrmin}{\mathrm{min}}
\newcommand{\nil}{\mathrm{nil}}
\newcommand{\mrP}{\mathrm{P}}
\newcommand{\mrpr}{\mathrm{pr}}
\newcommand{\proj}{\mathrm{Proj}}
\newcommand{\rk}{\mathrm{rk}}
\newcommand{\mrrk}{\mathrm{rk_{ss}}}
\newcommand{\mrrkp}{\mathrm{rk_{p}}}
\newcommand{\mrstab}{\mathrm{Stab}}
\newcommand{\mrrad}{\mathrm{rad}}
\newcommand{\mrsp}{\mathrm{Span}}
\newcommand{\mrV}{\mathrm{V}}
\newcommand{\HHG}{H^{\cdot}(G,\Bk)}
\newcommand{\HHE}{H^{\cdot}(E,\Bk)}
\newcommand{\HHcG}{H^{\cdot}(\mcG,\Bk)}
\newcommand{\Bk}{\Bbbk}
\newcommand{\iotaE}{\iota_{*,\mcE}(V_{r}(\mcE))}
\newcommand{\ck}{\mathrm{char}(\Bk)}
\newcommand{\Ap}{\Bk{\mbZ/p\mbZ}}
\newcommand{\ug}{\underline{\mfg}}
\newcommand{\cnil}{\mfC^{\nil}}

\def\Ab{\operatorname{Ab}\nolimits}
\def\Ad{\operatorname{Ad}\nolimits}
\def\ad{\operatorname{ad}\nolimits}
\def\d{\operatorname{\bf{d}}\nolimits}
\def\deg{\operatorname{deg}\nolimits}
\def\ev{\operatorname{ev}\nolimits}
\def\exp{\operatorname{exp}\nolimits}
\def\Ens{\operatorname{Ens}\nolimits}
\def\Ext{\operatorname{Ext}\nolimits}
\def\GL{\operatorname{GL}\nolimits}
\def\Gr{\operatorname{Gr}\nolimits}
\def\gr{\operatorname{gr}\nolimits}
\def\id{\operatorname{id}\nolimits}
\def\Im{\operatorname{Im}\nolimits}
\def\ker{\operatorname{ker}\nolimits}
\def\im{\operatorname{im}\nolimits}
\def\Ker{\operatorname{Ker}\nolimits}
\def\Mat{\operatorname{Mat}\nolimits}
\def\mult{\operatorname{mult}\nolimits}
\def\odd{\operatorname{odd}\nolimits}
\def\ord{\operatorname{ord}\nolimits}
\def\PSL{\operatorname{PSL}\nolimits}
\def\Rad{\operatorname{Rad}\nolimits}
\def\res{\operatorname{res}\nolimits}
\def\SL{\operatorname{SL}\nolimits}
\def\SO{\operatorname{SO}\nolimits}
\def\SP{\operatorname{SP}\nolimits}

\newenvironment{changemargin}[1]{%
  \begin{list}{}{%
    \setlength{\topsep}{0pt}%
    \setlength{\topmargin}{#1}%
    \setlength{\listparindent}{\parindent}%
    \setlength{\itemindent}{\parindent}%
    \setlength{\parsep}{\parskip}%
  }%
  \item[]}{\end{list}}

\parindent=0pt
\addtolength{\parskip}{0.5\baselineskip}

\title{Saturation rank for nilradical of parabolic subalgebras in Type A }
\author{Yang Pan}
\address{Wuxi Institute of Technology, Wuxi 214121, P.R.China} \email{ypan@outlook.de}

\pagestyle{plain}
\begin{abstract}
Let $\mfp(d)$ be a standard parabolic subalgebra of $\mfsl_{n+1}(K)$ and 
$\mfu$ be the corresponding nilradical defined over an algebraically 
closed field $K$ of characteristic $p>0$. We construct a finite connected quiver $Q(d)$, through which we provide a combinatorial characterization of the centralizer $c_{\mfu}(x(d))$ of the Richardson element $x(d)$.  
We specifically  focus on the centralizer  when the levi factor of $\mfp(d)$ is determined by either one or two simple roots.
This allows us to demonstrate that, under certain mild restrictions, the saturation rank of $\mfu$ equals the semisimple rank of the algebraic 
$K$-group $\SL_{n+1}(K)$.
\end{abstract}
\maketitle

\section{Introduction}
Let $(\mfg,[p])$ be a finite-dimensional restricted Lie algebra defined
over an algebraically closed field $K$ of characteristic $p>0$.
The restricted nullcone of $\mfg$ is the fiber of zero of the $[p]$-map, which is $V(\mfg)=\{ x\in \mfg \mid x^{[p]}=0\}$.
The subset
\[ 
\mbE(r,\mfg) : = \{\mfe \in \Gr_{r}(\mfg) \mid  [\mfe, \mfe]=0, \mfe \subset V(\mfg)\}
\]
of the Grassmannian $\Gr_{r}(\mfg)$ of $r$-planes, introduced in \cite{CFP} is closed and hence a projective variety.
We write the union of elements of $\mbE(r,\mfg)$ as 
$V_{\mbE(r,\mfg)}:= \bigcup_{\mfe \in \mbE(r, \mfg)}\mfe$, which is contained in the conical variety $V(\mfg)$.
We consider an important invariant of restricted Lie algebras $\mfg$, denoted as $\srk(\mfg)$, and call it the saturation rank.
This rank is defined by
\[ 
\srk(\mfg):=\mmax \{r \in \mbN \mid V(\mfg) = V_{\mbE(r,\mfg)}\}.
\]
When restricting a $\mathfrak{g}$-module to elements of $\mbE(r,\mfg)$  for a certain rank $r$ within a restricted Lie algebra $\mathfrak{g}$, it is crucial to ensure that no information is lost in comparison to its restricted nullcone $V(\mfg)$. This is the pivotal moment where the saturation rank takes center stage.
As demonstrated in \cite{Far}, it has been established that the Carlson module $L_{\zeta}$ remains indecomposable when the saturation rank of 
$\mathfrak{g}$ satisfies $\srk(\mathfrak{g})\geq 2$.
A prototypical case occurs when $\mfg$ is the algebraic Lie algebras of reductive algebraic groups $G$, which implies that
$\srk(\mfg)=\rk_{ss}(G)$ is the semisimple rank of a reductive algebraic group under some mild restrictions (cf. \cite{Pan}). In other cases, such 
as the algebraic Lie algebras of non-reductive groups, the saturation rank remains unknown.

Let $G$ be a reductive algebraic group over an algebraically closed field $K$ and let $P$ be a parabolic subgroup of $G$ with unipotent radical $U$. We write $\mfg$, $\mfp$ and $\mfu$ for the Lie algebras of $G$,$P$ and $U$ respectively. The established fact that $G$ exhibits finitely many nilpotent orbits within the Lie algebra $\mfg$ is widely acknowledged:this was initially proved by Richardson under the condition of either char$K$ is zero or good for $G$; we direct interested readers to consult \cite{Car} for an overview of the result in bad characteristic. 
It follows that there is a unique nilpotent orbit 
$G \cdot e$ which intersects $\mfu$ in an open dense subvariety. Richardson's dense orbit theorem tells us that the intersection 
$G\cdot e \cap \mfu = P\cdot e$ is a single $P$-orbit (we may assume that
$e\in \mfu$). The $P$-orbit $P \cdot e$ is called the Richardson orbit and its elements are called Richardson elements (cf. \cite{Ric}).

The purpose of this paper is to determine the saturation rank of the
nilpotent radical $\mfu$ in case $G=\SL_{n+1}(K)$. 
For any parabolic subgroup $P$ of $G$, there is a unique
dimension vector $d$ such that $P$ is conjugate to $P(d)$, where $d$
gives the sizes of the blocks in the Levi subgroup of $P$
containing maximal torus $T$. Therefore, in what follows it suffices to just consider parabolic subgroups of the form $P(d)$. 
In view of Lemma \ref{SD}, under some mild restrictions, the saturation rank $\srk(\mfu)$ is determined by the local saturation rank of Richardson elements.
The construction work of  such elements has been elucidated by Brüstle et al. in \cite{BHRR}. Furthermore, Baur et al. describe a normal form for Richardson elements in the classical case in \cite{Bau,BG}.

We write  $\mfp=\mfp(d)$ as the Lie algebra of $P(d)$ where $d =(d_{1},\ldots,d_{r})$ is the corresponding dimension vector. The Richardson element which is obtained through a horizontal line diagram $L_{h}(d)$ (see \cite{Bau} for more details) is now expressed by
\[
x(d)=x(L_{h}(d))=\sum_{i-j} e_{i,j}.
\]
We conclude this introduction with a succinct overview of the contents covered in our paper. In Section 2, we present an alternative characterization of the nilradical 
$\mfu$ of $\mfp(d)$ and introduce a method for identifying elements that commute with the Richardson element $x(d)$ for a given  dimension vector $d$. Section 3 deals with the  centralizer of  $x(d)$ when the Levi factor of $\mfp(d)$ is determined by either one or two simple roots. Finally we show that, in Section 4, the saturation rank of $\mfu$ coincides with the semisimple rank of the group $G=\SL_{n+1}(K)$  subject to certain constraints.
Consequently, we conclude that the Carlson module $L_{\zeta}$, acting as a module over $U_{0}(\mfu)$, remains indecomposable for $n\geq 2$ and when the characteristic of $K$ is greater than or equal to $n+1$.

{\bf Acknowledgement.} 
This work is supported by the National Natural Science Foundation of China(No.12201171).
The findings presented in this paper are derived from the author's doctoral thesis, which was undertaken at the University of Kiel. The author would
like to thank his advisor, Rolf Farnsteiner, for his continuous support.

\section{Finite quiver arising from the dimension vector }
Let $G=\SL_{n+1}(K)$ be the special linear group over an algebraically closed field $K$. Let $T$ be maximal torus of $G$ and $P(d)$ a standard 
parabolic subgroup of $G$ containing $T$. 
Let $\Phi$ be the root system of $G$ with respect to $T$, 
$\Delta=\{\alpha_1,\ldots,\alpha_n\}$ be the base of $\Phi$ and $\Phi^{+}$ be the set of positive roots.
We consider the  parabolic subalgebra $\mfp(d)=\mrlie(P(d))$ with its decomposition $\mfp(d)=\mfm\oplus \mfu$.
Let $\Phi(\mfm)\subseteq \Phi$ be the  closed subsystem of $\Phi$ determined by the levi factor $\mfm$  and 
$\Delta(\mfm)\subseteq \Delta$ be the base of $\Phi(\mfm)$. Then the nilpotent radical
$\mfu$ can be writen as $\mfu=\bigoplus_{\alpha \in \Phi^{+}\setminus \Phi(\mfm)^{+}}\mfg_{\alpha}$
where  $\mfg_{\alpha}$ is the root subspace of $\mfg$ corresponding to 
$\alpha$. In this section, our study centers on quivers $\Gamma:= (\Gamma_{0}, \Gamma_{1})$ that exhibit a lack of loops or multiple arrows where
$\Gamma_{0}=\{ 1,2,\dots, n+1\}$. For such quivers, the set of arrows is  represented as a subset $\Gamma_{1} \subseteq \Gamma_{0} \times \Gamma_{0}$, stemming from the Cartesian product of the set of vertices $\Gamma_{0}$, and 
two maps $s,t: \Gamma_{1}\mapsto \Gamma_{0}$ which associated to each arrow $\alpha\in \Gamma_{1}$ its source $s(\alpha)\in \Gamma_{0}$ and its target 
$t(\alpha)\in \Gamma_{0}$, respectively.

\begin{Defn}
Let $d=(d_{1},\ldots,d_{r})$ be a dimension vector associated to $\mfp(d)$. Arrange $r$ columns of $d_{i}$
dots, top-adjusted. 
Given $(a,b)\in \Gamma_{0}\times \Gamma_{0}$, there is an arrow 
$\alpha:a \rightarrow b $ if $a-b$ is a horizontal line in $L_{h}(d)$  or $a,b$ are from two adjacent columns and the column where
$a$ lies in is on the left. Let $\Gamma_{1}$ be the set of arrows, then 
$(\Gamma_{0},\Gamma_{1})$ is a locally connected finite quiver, denoted by $Q(d)$.
\end{Defn}

\begin{Exam}
Considering the parabolic subalgebra $\mfp(d)$ of $\mfsl_{9}(K)$ with dimension vector $d=(3,1,2,3)$, then $Q(d)$ is as follows:
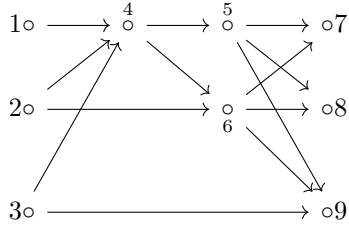
\[\begin{tikzcd}
	1\circ & \stackrel{4}{\circ} & \stackrel{5}{\circ} & \circ 7\\
	2\circ && \underset{6}{\circ} & \circ 8 \\
	3\circ &&& \circ 9
	\arrow[from=1-1, to=1-2]
	\arrow[from=1-2, to=1-3]
	\arrow[from=1-3, to=1-4]
	\arrow[from=2-1, to=1-2]
     \arrow[from=2-1, to=2-3]
	\arrow[from=3-1, to=1-2]
	\arrow[from=3-1, to=3-4]
	\arrow[from=1-3, to=2-4]
	\arrow[from=1-3, to=3-4]
	\arrow[from=2-3, to=1-4]
	\arrow[from=2-3, to=2-4]
	\arrow[from=2-3, to=3-4]
	\arrow[from=1-2, to=2-3]
\end{tikzcd}\]
\end{Exam}

\begin{Thm}{\label{map1}}
Let $\mfu$ be the nilpotent radical of a standard parabolic subalgebra $\mfp(d)$. There exists an  admissible ideal $J$ of path algebra $KQ(d)$ generated by all commutativity relations  $\omega_{1}-\omega_{2}$ such that $\mfu \cong radKQ(d)/ J$ as Lie algebras.
\end{Thm}
\begin{proof}
We first construct an algebra homomorphism 
\begin{eqnarray*}
       \varphi: \;\; \mrrad KQ(d)  & \longrightarrow & \mfu  \\
            \rho  & \mapsto & e_{s(\rho),t(\rho)}
\end{eqnarray*}
Recall that $\mfu=\bigoplus_{\alpha \in \Phi^{+}\setminus \Phi(\mfm)^{+}}\mfg_{\alpha}$, the indices $i,j$ of the root vector $x_{\alpha}=e_{i,j}$ 
for  $\alpha=\epsilon_{i}-\epsilon_{j}$ in $\Phi^{+}\setminus \Phi(\mfm)^{+}$ existing as vertices in $Q(d)$ is connected by a path $\rho$ since $i,j$ are from different columns.  We claim that $\varphi$ is well-defined  and surjective. Given two paths $\omega_{1}$ and $\omega_{2}$, if $\varphi(\omega_{1})=\varphi(\omega_{2})$, then $s(\omega_{1})=s(\omega_{2})$
and $t(\omega_1)=t(\omega_2)$.
Then the map $\varphi$ admits an admissible ideal generated by all commutativity relations $\omega_1-\omega_2$ as a kernel. Hence we conclude that the map $\varphi$ defined in the statement is an isomorphism.

The  Lie structure on $\mrrad KQ(d)$, which is defined as $[x,y]_{Q}=xy-yx$ for $x,y\in \mrrad KQ(d)$, where $xy$ and $yx$  are obtaind by the product of two paths in path algebra $KQ(d)$. It is obvious that $\varphi$ is compatible with Lie brackets, i.e. $\varphi[x,y]_{Q}=[\varphi(x), \varphi(y)]$, so the map $\varphi$ is an isomorphism as Lie algebras.
\end{proof}

\begin{Exam}
Let $d=(1,2,1)$ be the dimension vector and $Q(d)$ be the corresponding quiver
\[\begin{tikzcd}
	1\circ & \stackrel{2}{\circ} & \circ 4 \\
	& \underset{3}{\circ}
	\arrow["\alpha", from=1-1, to=1-2]
	\arrow["\beta", from=1-2, to=1-3]
	\arrow["\gamma"', from=1-1, to=2-2]
	\arrow["\delta"', from=2-2, to=1-3]
\end{tikzcd}\]
The $K$-algebra homomorphism $\varphi:  \mrrad KQ(d) \longrightarrow  \mfu$ is
defined by 
\begin{align*}
\varphi(\alpha)& =e_{1,2}, & \varphi(\beta)&=e_{2,4}, & \varphi(\gamma)&=e_{1,3},\\
\varphi(\delta)&=e_{3,4}, & \varphi(\alpha\beta)&=\varphi(\gamma\delta)=e_{1,4}.
\end{align*}
Here, we see that $\varphi$ is a surjection and  
$\Ker\;\varphi=<\alpha\beta- \gamma\delta>=J$. 
Hence, $\mrrad KQ(d)/J \cong \mfu$.
\end{Exam}

Given a vertex $x\in \Gamma_{0}$ in $Q(d)$, we put
\[
x^{+}:=\{ y\in  \Gamma_{0}\mid x\rightarrow y \;\mbox{is a horizontal arrow in}\; Q(d)\}
\]
\[
x^{-}:=\{ y\in \Gamma_{0} \mid  y\rightarrow x  \; \mbox{is a horizontal arrow in}\; Q(d) \},
\]
so that $x^{+}$ and $x^{-}$ are the subsets of successor and predecessor of the vertex $x$ in $Q(d)$, respectively. Obviously, 
$x^{+}$ (resp. $x^{-}$) is a singleton or an empty set, so we may identify the set $x^{+}$ (resp. $x^{-}$) with its element when it is non-empty without any ambiguous. If we have two vertices $a,b \in \Gamma_{0}$ in $Q(d)$ with $a<b$, we use $l(a)$(resp. $l(b)$) to indicate the line number in which $a$(resp. b) lies in the quiver $Q(d)$. If $l(a)=l(b)$, then $e_{a,b}$ is a summand of $x(d)$, which allows us to rewrite $x(d)=x_{1}+x_{2}+\cdots+x_{s}$ for some integer $s$
with 
\[
x_{i}=\sum_{\substack{a-b \\ l(a)=l(b)=i}}e_{a,b}
\]

\begin{Lem}{\label{sp}}
Let $x(d)$ be a Richardson element of $\mfp(d) $ written as $x(d)=x_{1}+x_{2}+\cdots+x_{s}$. Given an element
\[
x=\sum_{\epsilon_{i}-\epsilon_{j} \in  \Phi^{+}\setminus \Phi(\mfm)^{+} }k_{i,j}e_{i,j} 
\]
of $\mfu$ with  $ k_{a,b}\neq 0$  and $l(a)\neq l(b)$. If $[x(d),x]=0$, then we have the following three statements:
\begin{itemize}
\item[(1)] If $a^{-}\neq \emptyset$, then $b^{-}\neq \emptyset $ and $k_{a^{-},b^{-1}}=k_{a,b}$.
\item[(2)] If $b^{+}\neq \emptyset$, then $a^{+} \neq \emptyset $ and $k_{a^{+}, b^{+}}=k_{a,b}$.
\item[(3)] If $a^{-}=b^{+}=\emptyset$, then  $[x(d),x- k_{a,b}e_{a,b}]=0$.
\end{itemize}
\end{Lem}
\begin{proof}
We first have the statement: $[x_{i}, e_{a,b}]=0$ for $i \neq l(a),l(b)$. Then we have 
\[
[x(d),x]=[x_{l(a)}, k_{a,b}e_{a,b}] + [x_{l(b)}, k_{a,b}e_{a,b}]+[x(d),x^{'}]=0
\]
where $x^{'}=x-k_{a,b}e_{a,b}$.
If $a^{-}\neq \emptyset$, then $[x_{l(a)}, k_{a,b}e_{a,b}]=k_{a,b}e_{a^{-},b}$. Since $-k_{a,b}e_{a^{-}, b}$ cannot appear as a summand of  
$[x_{l(b)}, k_{a,b}e_{a,b}]$, which enforces the term $k_{a,b}e_{a^{-},b}$ existing as a summand of $x^{'}\cdot x(d)$.
Specifically, we should have $k_{a^{-},c}e_{a^{-},c} \cdot  e_{c,b}=k_{a,b}e_{a^{-1},b}$ for some $c$ where $k_{a^{-},c}e_{a^{-},c}$ is a term of
$x^{'}$ and $e_{c,b}$ is a term of $x(d)$. By the construction of $x(d)$,  we have $c=b^{-}$. Therefore, $b^{-}\neq \emptyset $ and $k_{a^{-},b^{-}}=k_{a,b}$,
which proves statement (1).

If $b^{+}\neq \emptyset$,  then $[x_{l(b)}, k_{a,b}e_{a,b}]=-k_{a,b}e_{a,b^{+}}$. By the same token, we  shall have  $a^{+}\neq \emptyset$ and 
$k_{a^{+}, b^{+}}=k_{a,b}$. Instead of offering the proof of assertion (2), we opt to present a diagram in case $l(a) < l(b)$.

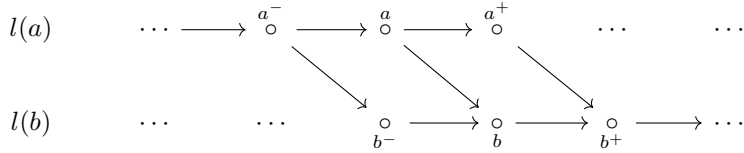
\[\begin{tikzcd}
	{l(a)} & \cdots & {\stackrel{a^{-}}{\circ}} & {\stackrel{a}{\circ}} & {\stackrel{a^{+}}{\circ}} & \cdots & \cdots \\
	{l(b)} & \cdots & \cdots & {\underset{b^{-}}{\circ}} & {\underset{b}{\circ}} & {\underset{b^{+}}{\circ}} & \cdots
	\arrow[from=1-3, to=1-4]
	\arrow[from=1-4, to=2-5]
	\arrow[from=1-3, to=2-4]
	\arrow[from=2-4, to=2-5]
	\arrow[from=1-2, to=1-3]
	\arrow[from=1-4, to=1-5]
	\arrow[from=2-5, to=2-6]
	\arrow[from=1-5, to=2-6]
	\arrow[from=2-6, to=2-7]
\end{tikzcd}\]
If $a^{-}=b^{+}=\emptyset$, then $[x_{l(a)}, k_{a,b}e_{a,b}]= [x_{l(b)}, k_{a,b}e_{a,b}]=0$, implying $[x(d),x^{'}]=0$, which proves statement (3).
\end{proof}

\begin{remark}{\label{r1}}
Theorem \ref{map1} and Lemma \ref{sp} provides us with a method to identify all elements that commute with the Richardson element $x(d)$. For example, let's consider the parabolic subalgebra $\mfp(d)\subseteq \mfsl_{9}(K)$ defined by the dimension vector $d=(3,2,3,1)$. In this case, $\Delta(\mfm)=\{\alpha_{1},\alpha_{2},\alpha_{4},\alpha_{6},\alpha_{7}\}$.
Apart from powers of the Richardson element $x(d)$, the following ten diagrams provide the generators in $\mfu$ that commute with $x(d)$:
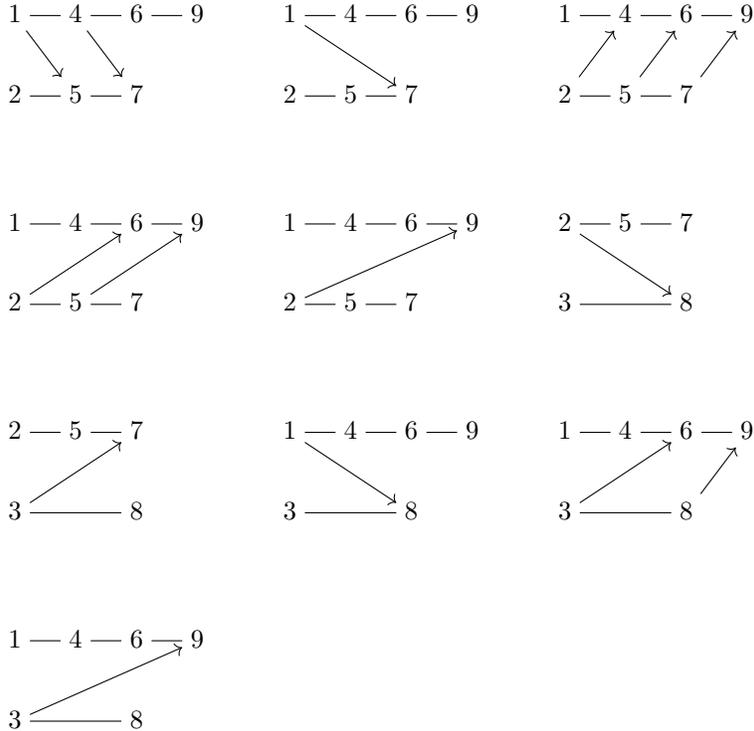
\[
\begin{tikzcd}[ampersand replacement=\&,cramped,column sep=small]
	1 \& 4 \& 6 \& 9 \&\& 1 \& 4 \& 6 \& 9 \&\& 1 \& 4 \& 6 \& 9 \\
	2 \& 5 \& 7 \&\&\& 2 \& 5 \& 7 \&\&\& 2 \& 5 \& 7 \\
	\\
	1 \& 4 \& 6 \& 9 \&\& 1 \& 4 \& 6 \& 9 \&\& 2 \& 5 \& 7 \\
	2 \& 5 \& 7 \&\&\& 2 \& 5 \& 7 \&\&\& 3 \&\& 8 \\
	\\
	2 \& 5 \& 7 \&\&\& 1 \& 4 \& 6 \& 9 \&\& 1 \& 4 \& 6 \& 9 \\
	3 \&\& 8 \&\&\& 3 \&\& 8 \&\&\& 3 \&\& 8 \\
	\\
	1 \& 4 \& 6 \& 9 \\
	3 \&\& 8
	\arrow[no head, from=1-1, to=1-2]
	\arrow[from=1-1, to=2-2]
	\arrow[no head, from=1-2, to=1-3]
	\arrow[from=1-2, to=2-3]
	\arrow[no head, from=1-3, to=1-4]
	\arrow[no head, from=1-6, to=1-7]
	\arrow[from=1-6, to=2-8]
	\arrow[no head, from=1-7, to=1-8]
	\arrow[no head, from=1-8, to=1-9]
	\arrow[no head, from=1-11, to=1-12]
	\arrow[no head, from=1-12, to=1-13]
	\arrow[no head, from=1-13, to=1-14]
	\arrow[no head, from=2-1, to=2-2]
	\arrow[no head, from=2-2, to=2-3]
	\arrow[no head, from=2-6, to=2-7]
	\arrow[no head, from=2-7, to=2-8]
	\arrow[from=2-11, to=1-12]
	\arrow[no head, from=2-11, to=2-12]
	\arrow[from=2-12, to=1-13]
	\arrow[no head, from=2-12, to=2-13]
	\arrow[from=2-13, to=1-14]
	\arrow[no head, from=4-1, to=4-2]
	\arrow[no head, from=4-2, to=4-3]
	\arrow[no head, from=4-3, to=4-4]
	\arrow[no head, from=4-6, to=4-7]
	\arrow[no head, from=4-7, to=4-8]
	\arrow[no head, from=4-8, to=4-9]
	\arrow[no head, from=4-11, to=4-12]
	\arrow[from=4-11, to=5-13]
	\arrow[no head, from=4-12, to=4-13]
	\arrow[from=5-1, to=4-3]
	\arrow[no head, from=5-1, to=5-2]
	\arrow[from=5-2, to=4-4]
	\arrow[no head, from=5-2, to=5-3]
	\arrow[from=5-6, to=4-9]
	\arrow[no head, from=5-6, to=5-7]
	\arrow[no head, from=5-7, to=5-8]
	\arrow[no head, from=5-11, to=5-13]
	\arrow[no head, from=7-1, to=7-2]
	\arrow[no head, from=7-2, to=7-3]
	\arrow[no head, from=7-6, to=7-7]
	\arrow[from=7-6, to=8-8]
	\arrow[no head, from=7-7, to=7-8]
	\arrow[no head, from=7-8, to=7-9]
	\arrow[no head, from=7-11, to=7-12]
	\arrow[no head, from=7-12, to=7-13]
	\arrow[no head, from=7-13, to=7-14]
	\arrow[from=8-1, to=7-3]
	\arrow[no head, from=8-1, to=8-3]
	\arrow[no head, from=8-6, to=8-8]
	\arrow[from=8-11, to=7-13]
	\arrow[no head, from=8-11, to=8-13]
	\arrow[from=8-13, to=7-14]
	\arrow[no head, from=10-1, to=10-2]
	\arrow[no head, from=10-2, to=10-3]
	\arrow[no head, from=10-3, to=10-4]
	\arrow[from=11-1, to=10-4]
	\arrow[no head, from=11-1, to=11-3]
\end{tikzcd}
\]
The corresponding elements, from left to right and from top to bottom, are
\begin{align*}
e_{1,5}+e_{4,7} && e_{1,7} && e_{2,4}+e_{5,6}+e_{7,9} && e_{2,6}+e_{5,9} && e_{2,9}& \\
e_{2,8} & &  e_{3,7} &&  e_{1,8} &&  e_{3,6}+e_{8,9} && e_{3,9} & 
\end{align*}
\end{remark}
\section{Centralizers of Richardson elements with $|\Delta(\mfm)|\leq 2$}
Given $x(d)$ as a Richardson element of $\mfp(d)$,
through the natural embedding of $\mfsl_{n+1}(K)=\mrlie(\SL_{n+1}(K))$ in 
$\gl(\mbV)$ with $\dim \mbV = n+1$, each Richardson element 
$x(d)$ becomes a nilpotent element in
$\gl(\mbV)$, and therefore a nilpotent endomorphism of $\mbV$. 
We can therefore associate to $x(d)$ a partition $\pi$, written
sometimes in the form $(\lambda_{1}\geq \lambda_{2}\geq \cdots\geq \lambda_{r}>0)$, and sometimes as  $[1^{r_{1}}2^{r_{2}}3^{r_{3}}\ldots]$,
given by the sizes of the blocks in the Jordan normal form of $x(d)$.

\begin{Lem}{\label{JF}}
The partition of the Richardson element $x(d)$ is, up to a permutation,
\begin{itemize}
\item[(1)] $\pi=[1,n]$ provided $|\Delta(\mfm)|=1$.
\item[(2)] $\pi=[2, n-1]$ or $\pi=[1^{2},n-1]$ provided $|\Delta(\mfm)|=2$.
\end{itemize}
\end{Lem}
\begin{proof}
(1) Since $|\Delta(\mfm)|=1$, we may assume that 
$\Delta(\mfm)=\{ \alpha_{i} \}$ for some $i$. 
Then the associated dimension vector is $d=(1,\ldots,2,\ldots,1)$, with the $i$-th coordinate is 2, and the remaining coordinates are 1. The number of  edges in the horizontal line diagram $L_{h}(d)$ is $n-1$, which means the corresponding Richardson element $x(d)$ is the sum of $n-1$ elementary matrices. As a result, the rank of $x(d))$ is $n-1$.
Since the eigenvalue of $x(d)$ is 0, the number of Jordan blocks is the geometric multiplicity of 0, which is $n+1-\rank(x(d))=2$.

The rank of $x(d)^{i}$ is $n-i$ for $1\leq i \leq n$, as well as 
$\rank(x(d)^{0})=n+1$ and $\rank(x(d)^{n+1})=0$. We have
$$ \rank(x(d)^{m-1})-2\rank(x(d)^{m})+\rank(x(d)^{m+1})=
         \begin{cases}
                 1,      &\mbox{if}\; m=1 \;\mbox{or}\; n,   \\
                 0,      & \mbox{otherwise}.\\              
         \end{cases}$$
Thus, the partition of $x(d)$ is $\pi=[1, n]$ proving the first statement.

(2) If $|\Delta(\mfm)|=2$, then the rank of $x(d)$ is either $n-1$ or $n-2$. The remaining proof is similar to part (1).
\end{proof}

If $\mid \Delta(\mfm) \mid =q$, then there exists an injective
map from set of 
labeled horizontal line diagrams $L_{h}(d)$ to  set $\mathbb{N}^{n+1-q}$,
defined by
\begin{eqnarray*}
\varphi: \{\mbox{linear diagram}\; L_{h}(d) \}  &\longrightarrow& \mathbb{N}^{n+1-q}\\
  L_{h}(d)  &\mapsto&  (D_{1},D_{2},\ldots, D_{n+1-q})
\end{eqnarray*}
where $D_{1}=1$ and $D_{i}=D_{1}+\sum_{j=1}^{i-1}d_{j}$ for $1<i\leq n+1-q$.
Thereafter, we may write 
$L_{h}(d)=(D_{1},D_{2},\cdots, D_{n+1-q}\mid \widetilde{D_{1}},\widetilde{D_{2}},\cdots,\widetilde{D_{q}})$ alternatively where
$\widetilde{D_{j}}$ for $1\leq j\leq q$ represents the remaining labeled numbers arranged from left to right and proceeds in a top-down manner. 

Recall that if $x(d)$ has partition 
$(\lambda_{1}\geq \lambda_{2}\geq \cdots\geq \lambda_{r}>0)$, then there exist $v_{1},v_{2},\ldots,v_{r}\in \mbV$ such that all $x(d)^{j}v_{i}$ with $1\leq i\leq r$ and $0\leq j < \lambda_{i}$ are a basis for $\mbV$ and such
that $x(d)^{\lambda_{i}}v_{i}=0$ for all $i$.
For each integer $m\geq 1$, denote by $J_{m}$ the $(m\times m)$-matrix where the $(i,i+1)$ entries with $1\leq i < m$ are equal to 1 and all remaining entries are  equal to 0.
In what follows, we shall give the $v_{i}$ for certain $x(d)$.

\begin{Lem}{\label{BC}}
Let $x(d)$ be a Richardson element of $\mfp(d)$ determined by $\Delta(\mfm)$.
\begin{itemize}
\item[(1)] If $\Delta(\mfm)=\{\alpha_{r}\}$ and $L_{h}(d)=(D_{1},D_{2},\cdots, D_{n}\mid \widetilde{D_{1}})$, then $\mbV$ has a basis
\[
\{ x(d)^{n-1}v_{1},\ldots,x(d)v_{1},v_{1},v_{2} \}
\]
where $v_{1}=e_{D_n},v_{2}=e_{\widetilde{D_{1}}}$ and the action of $x(d)$
on $v_{1}$ can be represented as a diagram like
\[\begin{tikzcd}
	{e_{D_{1}}} & {e_{D_{2}}} & \cdots & {e_{D_{n-1}}} & {e_{D_{n}}} \\
	{x(d)^{n-1}v_{1}} & {x(d)^{n-2}v_{1}} & \cdots & {x(d)v_{1}} & {v_{1}}
	\arrow["{x(d)}", shift right=2, curve={height=-12pt},from=1-5,to=1-4]
	\arrow["{x(d)}", shift right=2,curve={height=-12pt}, from=1-4, to=1-3]
	\arrow["{x(d)}", shift right=2,curve={height=-12pt}, from=1-3, to=1-2]
	\arrow["{x(d)}", shift right=2,curve={height=-12pt}, from=1-2, to=1-1]
\end{tikzcd}\]
\item[(2)] If $\Delta(\mfm)=\{ \alpha_{r},\alpha_{r+1}\}$ and 
$L_{h}(d)=(D_{1},D_{2},\cdots, D_{n-1}\mid \widetilde{D_{1}},\widetilde{D_{2}})$, then $\mbV$ has a basis
\[
\{ x(d)^{n-2}v_{1},\ldots,x(d)v_{1},v_{1},v_{2},v_{3} \}
\]
where $v_{1}=e_{D_{n-1}},v_{2}=e_{\widetilde{D_{1}}},v_{3}=e_{\widetilde{D_{2}}}$ and the action of $x(d)$ on $v_{1}$ can be represented as a diagram like
\[\begin{tikzcd}
	{e_{D_{1}}} & {e_{D_{2}}} & \cdots & {e_{D_{n-2}}} & {e_{D_{n-1}}} \\
	{x(d)^{n-2}v_1} & {x(d)^{n-3}v_1} & \cdots & {x(d)v_1} & {v_{1}}
	\arrow["{x(d)}", shift right=2, curve={height=-12pt}, from=1-5, to=1-4]
	\arrow["{x(d)}", shift right=2, curve={height=-12pt}, from=1-4, to=1-3]
	\arrow["{x(d)}", shift right=2, curve={height=-12pt}, from=1-3, to=1-2]
	\arrow["{x(d)}", shift right=2, curve={height=-12pt}, from=1-2, to=1-1]
\end{tikzcd}\]
\item[(3)] If $\Delta(\mfm)=\{ \alpha_{r},\alpha_{s}\}$ with $r<s-1$ and 
$L_{h}(d)=(D_{1},D_{2},\cdots, D_{n-1}\mid \widetilde{D_{1}},\widetilde{D_{2}})$, then $\mbV$ has a basis
\[
\{ x(d)^{n-2}v_{1},\ldots,x(d)v_{1},v_{1},x(d)v_{2},v_{2} \}
\]
where $v_{1}=e_{D_{n-1}},v_{2}=e_{\widetilde{D_{2}}}$ and the action of $x(d)$ on $v_{1},v_2$ can be represented as a diagram like
\[\begin{tikzcd}
	{e_{D_{1}}} & {e_{D_{2}}} & \cdots & {e_{D_{n-2}}} & {e_{D_{n-1}}} \\
	{x(d)^{n-2}v_1} & {x(d)^{n-3}v_1} & \cdots & {x(d)v_1} & {v_{1}}
	\arrow["{x(d)}", shift right=2, curve={height=-12pt}, from=1-5, to=1-4]
	\arrow["{x(d)}", shift right=2, curve={height=-12pt}, from=1-4, to=1-3]
	\arrow["{x(d)}", shift right=2, curve={height=-12pt}, from=1-3, to=1-2]
	\arrow["{x(d)}", shift right=2, curve={height=-12pt}, from=1-2, to=1-1]
\end{tikzcd}\]
\[\begin{tikzcd}
	{e_{\widetilde{D_{1}}}} & {e_{\widetilde{D_{2}}}} \\
	{x(d)v_2} & {v_2}
	\arrow["{x(d)}", shift right=2, curve={height=-12pt}, from=1-2, to=1-1]
\end{tikzcd}\]
\end{itemize}
\end{Lem}
\begin{proof}
We exclusively demonstrate the veracity of statement (1), noting the analogous nature of proofs for statements (2) and (3).
Assume $x(d)$ is the Richardson element given by $\Delta(\mfm)=\{\alpha_{r}\}$ and $L_{h}(d)=(D_{1},D_{2},\cdots, D_{n}\mid \widetilde{D_{1}})$, then  
$x(d)={\displaystyle \sum_{i=1}^{n-1}e_{D_i,D_{i+1}}}$.
Given such a diagram $L_{h}(d)$, we let $\sigma$ be the permutation of the 
set $\Gamma_{0}$, where 
$$ \sigma({i})=
         \begin{cases}
                 D_{i},    &\mbox{if}\; 1\leq i \leq n ,   \\
   \widetilde{D_{1}},      & \mbox{if}\; i= n+1. \\              
         \end{cases}$$
We now define $P={\displaystyle \prod_{i=1}^{n} E(i,\sigma(i))}$ as the product of elementary matrices. Then we have $P=(e_{\sigma(1)},e_{\sigma(2)},\ldots,e_{\sigma(n+1)})=(e_{D_1},e_{D_2},\ldots, e_{D_{n}},e_{\widetilde{D_{1}}})$ and further 
\[
x(d)(e_{\sigma(1)},e_{\sigma(2)},\ldots,e_{\sigma(n+1)})=
(e_{\sigma(1)},e_{\sigma(2)},\ldots,e_{\sigma(n+1)})
\begin{pmatrix}
   J_{n} & 0 \\
       0 & J_{1}
 \end{pmatrix}, \]
implying $x(d)e_{\sigma(i+1)}=e_{\sigma(i)}$ for $1\leq i \leq n-1$ 
and $x(d)e_{\sigma(n+1)}=0$.
As a result, we have $v_{1}=e_{\sigma(n)}=e_{D_n}, v_{2}=e_{\sigma(n+1)}=e_{\widetilde{D_{1}}}$ and the action of $x(d)$ on $v_{1}$ described in the diagram, as desired.
\end{proof}

\begin{Exam}
Consider the restricted Lie algebra $\mfg=\mfsl_{6}(K)$. We give three examples to illustrate Lemma \ref{BC}:
\begin{itemize}
\item[(1)] Suppose $\Delta(\mfm)=\{ \alpha_{3}\}$. Then the dimension vector
is $d=(1,1,2,1,1)$ and the line diagram $L_{h}(d)$
\[\begin{tikzcd}[cramped,sep=scriptsize]
	1 & 2 & 3 & 5 & 6 \\
	&& 4
	\arrow[no head, from=1-1, to=1-2]
	\arrow[no head, from=1-2, to=1-3]
	\arrow[no head, from=1-3, to=1-4]
	\arrow[no head, from=1-4, to=1-5]
\end{tikzcd}\]
can be written as $(1,2,3,5,6 \mid 4)$.
Let $\sigma=(456)$ be the permutation, then
Lemma \ref{BC} provides that $v_1=e_{\sigma(5)}=e_6, x(d)^{i}v_{1}=e_{\sigma(5-i)}(1\leq i\leq 4) $ and $v_{2}=e_{\sigma(6)}=e_{4}$.
\item[(2)]Let $\Delta(\mfm)=\{ \alpha_{2},\alpha_{3}\}$. Then the dimension vector is $d=(1,3,1,1)$ and the line diagram $L_{h}(d)$ 
\[\begin{tikzcd}[cramped,sep=scriptsize]
	1 & 2 & 5 & 6 \\
	& 3 \\
	& 4
	\arrow[no head, from=1-1, to=1-2]
	\arrow[no head, from=1-2, to=1-3]
	\arrow[no head, from=1-3, to=1-4]
\end{tikzcd}\]
can be written as $(1,2,5,6\mid 3,4)$.
Let $\sigma=(35)(46)$ be the permutation, then
Lemma \ref{BC} provides that 
$v_{1}=e_{\sigma(4)}=e_{6},  x(d)v_1=e_5, x(d)^{2}v_{1}=e_{2},x(d)^{3}v_{1}=e_{1}, v_{2}=e_{\sigma(5)}=e_{3}$ and
$v_{3}=e_{\sigma(6)}=e_{4}$.
\item[(3)] Let $\Delta(\mfm)=\{ \alpha_{1},\alpha_{3}\}$. Then the dimension vector is $d=(2,2,1,1)$ and the line diagram $L_{h}(d)$ 
\[\begin{tikzcd}[cramped,sep=scriptsize]
	1 & 3 & 5 & 6 \\
	2 & 4
	\arrow[no head, from=1-1, to=1-2]
	\arrow[no head, from=1-2, to=1-3]
	\arrow[no head, from=1-3, to=1-4]
	\arrow[no head, from=2-1, to=2-2]
\end{tikzcd}\]
can be written as $(1,3,5,6 \mid 2,4)$.
Let $\sigma=(23564)$ be the permutation, by Lemma \ref{BC}, we have 
$v_{1}=e_{\sigma(5)}=e_{6}, x(d)v_{1}=e_{5}, x(d)^{2}v_{1}=e_{3}, x(d)^{3}v_{1}=e_{1}, 
v_{2}=e_{\sigma(6)}=e_{4}$ and $x(d)v_{2}=e_{2}$.
\end{itemize}
\end{Exam}

Let $c_{\gl(\mbV)}(x(d))$ be the centralizer of $x(d)$ in $\gl(\mbV)$,
and $c_{\mfu}(x(d)):=c_{\gl(\mbV)}(x(d))\cap \mfu$ be the centralizer of $x(d)$ in $\mfu$.
Each $Z\in c_{\gl(\mbV)}(x(d))$ is determined by the $Z(v_{i})$ for $1\leq i \leq r$ because 
$Z(x(d)^{k}v_{i})=x(d)^{k}Z(v_{i})$ for all $i$ and $k$. Further we have to have $x(d)^{\lambda_{i}}Z(v_{i})=0$ for all $i$. Using this one checks
that $Z(v_{i})$ has the form
\[
 Z(v_{i})=\sum_{j=1}^{r}\sum_{k=\max(0,\lambda_{j}-\lambda_{i})}^{\lambda_{j}-1} a_{k,j;i}x(d)^{k}v_{j}.
\]
When $|\Delta(\mfm)|=0$, then $x(d)$ is the regular nilpotent element of 
$\mfsl_{n+1}(K)$. We refer the interested reader to \cite{Pan} for 
relevant results.
In the following, we will determine the centralizer 
$c_{\mfu}(x(d))$ of $x(d)$ and its center
$Z(c_{\mfu}(x(d)))$ when $1\leq |\Delta(\mfm)|\leq 2$.

\begin{Thm}{\label{Cent1}}
If $|\Delta(\mfm)|=1$,  the centralizer $c_{\mfu}(x(d))$ is characterized as
follows:
\begin{itemize}
\item[(1)] If $\Delta(\mfm)=\{ \alpha_{1}\}$, then elements 
\[
 e_{2,n+1},\;\; x(d)^{k}(1\leq k \leq n-1)
\]
form a basis of $c_{\mfu}(x(d))$.
\item[(2)] If $\Delta(\mfm)=\{ \alpha_{s} \}$ with $1<s<n$, then elements 
\[
 e_{1,s+1},\;\;e_{s+1,n+1},\;\; x(d)^{k}(1\leq k \leq n-1)
\]
form a basis of $c_{\mfu}(x(d))$.
\item[(3)] If $\Delta(\mfm)=\{\alpha_{n}\}$, then elements 
\[
e_{1,n+1}, \;\; x(d)^{k}(1\leq k \leq n-1)
\]
form a basis of $c_{\mfu}(x(d))$.
\end{itemize}
\end{Thm}
\begin{proof}
When $|\Delta(\mfm)|=1$, the partition of $x(d)$ is $[1,n]$ according to Lemma \ref{JF}. Then there exist two elements $v_{1}$ and $v_{2}$ such that
$x(d)^{i}v_{1}$ for $0\leq i\leq n-1$ together with $v_{2}$ form a basis 
of $\mbV$. Let $Z\in c_{\mfu}(x(d)) $. The first observation of 
$c_{\mfu}(x(d))=c_{\gl(\mbV)}(x(d)) \cap \mfu$, implying
$Z\in c_{\gl(\mbV)}(x(d))$ and further
\[
 Z(v_{1})=\sum_{k=0}^{n-1}a_{k,1;1}x(d)^{k}v_{1} + a_{0,2;1}v_{2}
\]
\[
Z(v_{2})= a_{n-1,1;2} x(d)^{n-1}v_{1}+ a_{0,2;2}v_{2}.
\]
Additionally, the fact that $Z\in \mfu$ gives $a_{0,1;1}=0$ and $a_{0,2;2}=0$. By assumption, we let  $\Delta(\mfm)=\{ \alpha_{s} \}$ for $1\leq s \leq n$. According to  Lemma \ref{BC}, we have
$$ v_{1}=
         \begin{cases}
                 e_{n+1},    &\mbox{if}\; 1\leq s <n ,   \\
                 e_{n},      & \mbox{if}\; s=n\\              
         \end{cases}$$
and $ v_{2}=e_{s+1}$.
\begin{case}
$\Delta(\mfm)=\{\alpha_{1}\}$. Since $x(d)^{n-1}v_{1}=e_{1}$, which gives $a_{n-1,1;2}=0$. Let $Z_{1}\in c_{\mfu}(x(d))$ with $Z_1(v_1)=v_{2}$.
Then $Z_1=e_{2,n+1}+Z_{1}^{'}$. By Lemma \ref{sp}, $Z_{1}^{'}\in c_{\mfu}(x(d))$, implying that $Z_{1}^{'}(v_{1})=Z_{1}^{'}(v_2)=0$, so 
$Z_{1}^{'}=0$. Therefore,
\[
Z=\sum_{k=1}^{n-1}a_{k,1;1}x(d)^{k}+ a_{0,2;1}e_{2,n+1}.
\]
\end{case}
\begin{case}
$\Delta(\mfm)=\{\alpha_{s}\}$ for $1<s<n+1$. Let $Z_{2} \in c_{\mfu}(x(d))$
with $Z_{2}(v_1)=v_{2}$, then $Z_2=e_{s+1,n+1}+Z_{2}^{'}$.   
By Lemma \ref{sp}, $Z_{2}^{'}\in c_{\mfu}(x(d))$, giving $Z_{2}^{'}(v_{1})=0$. If $Z_{2}^{'}(v_{2})=e_{1}$, then $Z_{2}^{'}=e_{1,s+1}$. 
Hene,
\[
Z=\sum_{k=1}^{n-1}a_{k,1;1}x(d)^{k}+ a_{0,2;1}e_{s+1,n+1}+a_{n-1,1;2}e_{1,s+1}.
\]
\end{case}

\begin{case}
$\Delta(\mfm)=\{\alpha_{n}\}$. Given that $v_{1}=e_{n}$ and $v_{2}=e_{n+1}$,  it can be deduced that $a_{0,2;1}=0$.  Let $Z_{3}\in c_{\mfu}(x(d))$ with $Z_3(v_2)=x(d)^{n-1}v_{1}=e_{1}$. Then $Z_{3}=e_{1,n+1}+Z_{3}^{'}$ and 
$Z_{3}^{'}\in c_{\mfu}(x(d))$ by Lemma \ref{sp}.Hence,
\[
Z=\sum_{k=1}^{n-1}a_{k,1;1}x(d)^{k}+a_{n-1,1;2}e_{1,n+1}.
\]
\end{case}

\end{proof}

\begin{Thm}{\label{Cent2}}
If $|\Delta(\mfm)|=2$ and $x(d)$ is of partition $[1^{2},n-1]$, then
the centralizer $c_{\mfu}(x(d))$ is characterized as follows:
\begin{itemize}
\item[(1)] If $\Delta(\mfm)=\{\alpha_{1},\alpha_{2}\}$, 
then elements
\[
e_{2,n+1},\;\; e_{3,n+1},\;\;x(d)^{k}(1\leq k \leq n-2)
\]
form a basis for  $c_{\mfu}(x(d))$.
\item[(2)] If $\Delta(\mfm)=\{\alpha_{r},\alpha_{r+1}\}$ with $1<r<n-1$, 
then elements
\[
e_{1,r+1},\;\; e_{1,r+2},\;\; e_{r+1,n+1},\;\; e_{r+2,n+1},\;\; 
x(d)^{k}(1\leq k \leq n-2)
\]
form a basis for  $c_{\mfu}(x(d))$.
\item[(3)] If $\Delta(\mfm)=\{\alpha_{n-1},\alpha_{n}\}$, 
then elements 
\[
e_{1,n},\;\; e_{1,n+1},\;\; x(d)^{k}(1\leq k \leq n-2)
\]
form a basis for $x_{\mfu}(x(d))$.
\end{itemize}
\end{Thm}
\begin{proof}
If $|\Delta(\mfm)|=2$ and the partition of $x(d)$ is $[1^{2},n-1]$,
then there exist three elements $v_{1}, v_{2}$ and $v_{3}$ such that
$x(d)^{i}v_{1}$ for $0\leq i\leq n-2$ together with $v_{2},v_{3}$ are a basis of $\mbV$. 
Let $Z\in c_{\mfu}(x(d))$. Then $Z\in c_{\gl(\mbV)}(x(d))$
and 
\begin{align*}
  Z(v_{1})=& \sum_{k=0}^{n-2}a_{k,1;1}x(d)^{k}v_{1}+a_{0,2;1}v_{2}+a_{0,3;1}v_{3} \\
Z(v_{2})=& a_{n-2,1;2}x(d)^{n-2}v_{1}+a_{0,2;2}v_{2}+a_{0,3;2}v_{3} \\
Z(v_{3})=& a_{n-2,1;3}x(d)^{n-2}v_{1}+ a_{0,2;3}v_{2}+a_{0,3;3}v_{3}.
\end{align*}
Since $Z\in \mfu$, we have $a_{0,1;1}=a_{0,2;2}=a_{0,3;3}=0$. 
By assumption, we may let $\Delta(\mfm)=\{\alpha_{r},\alpha_{r+1}\}$ for
$1\leq r <n$. By Lemma \ref{BC}, we have
$$ v_{1}=
         \begin{cases}
                 e_{n-1},      &\mbox{if} \; r=n-1 ,   \\
                 e_{n+1},      & \mbox{if}\; r<n-1,     
         \end{cases}$$
$v_{2}=e_{r+1}$ and $v_{3}=e_{r+2}$. The remainder of the proof exhibits similarity to the demonstration of Theorem \ref{Cent1}.
\end{proof}

\setcounter{case}{0}
\begin{Thm}{\label{Cent3}}
If $|\Delta(\mfm)|=2$ and $x(d)$ is of partition $[2,n-1]$,
then the centralizer $c_{\mfu}(x(d))$ is characterized as follows:
\begin{itemize}
\item[(1)] If $\Delta(\mfm)=\{\alpha_{1},\alpha_{n}\}$, 
then elements 
\[
e_{1,n+1},\;\;e_{2,n},\;\; x(d)^{k}(1\leq k \leq n-2)
\]
form a basis for $c_{\mfu}(x(d))$.
\item[(2)] If $\Delta(\mfm)=\{\alpha_{1},\alpha_{s}\}$ with $2<s<n$, 
then elements
\[
e_{1,s+1},\;\;e_{2,n+1},\;\; e_{2,n}+e_{s+1,n+1},\;\; x(d)^{k}(1\leq k \leq n-2)
\]
form a basis for  $c_{\mfu}(x(d))$.
\item[(3)] If $\Delta(\mfm)=\{\alpha_{r},\alpha_{n}\}$ with $1<r<n-1$, 
then elements
\[
e_{1,n+1},\;\;e_{r+1,n},\;\; e_{1,r+1}+e_{2,n+1},\;\; x(d)^{k}(1\leq k \leq n-2)
\]
form a basis for  $c_{\mfu}(x(d))$.
\item[(4)] If $\Delta(\mfm)=\{\alpha_{r},\alpha_{s}\}$ with $1<r<s-1<n-1$, 
then elements
\[
e_{1,s+1},\;\; e_{r+1,n+1},\;\; e_{1,r+1}+e_{2,s+1},\;\;e_{r+1,n}+e_{s+1,n+1},\;\; x(d)^{k}(1\leq k \leq n-2)
\]
form a basis for  $c_{\mfu}(x(d))$.
\end{itemize}
\end{Thm}
\begin{proof}
When $x(d)$ has partition $[2,n-1]$, 
there exist two elements $v_{1}$ and $v_{2}$ such that
$x(d)^{i}v_{1}$ for $0\leq i\leq n-2$ together with $v_{2}, x(d)v_{2}$ 
form a basis of $\mbV$. If $Z\in c_{\mfu}(x(d))$, then 
\begin{align*}
  Z(v_{1})=& \sum_{k=1}^{n-2}a_{k,1;1}x(d)^{k}v_{1}+a_{0,2;1}v_{2}+a_{1,2;1}x(d)v_{2} \\
Z(v_{2})=& a_{n-3,1;2}x(d)^{n-3}v_{1}+a_{n-2,1;2}x(d)^{n-2}v_{1}
+a_{1,2;2}x(d)v_{2} 
\end{align*}
with $a_{1,1;1}=a_{1,2;2}$.
Since $|\Delta(\mfm)|=2$, we may assume that 
$\Delta(\mfm)=\{\alpha_{r},\alpha_{s}\}$ where $1\leq r<s\leq n$ and $r+1<s$. 
By Lemma \ref{BC}, we have 
\begin{align*}
 v_{1}=
         \begin{cases}
               e_{n},        &\mbox{if}\; s=n ,   \\ 
               e_{n+1},      & \mbox{if}\; s<n,     
         \end{cases}
\end{align*}
and $v_{2}=e_{s+1}$.
\begin{case}
    $\Delta(\mfm)=\{\alpha_{1}, \alpha_{n} \}$. Then $a_{0,2;1}=a_{n-3,1;2}=0$ . Let $Z^{'}\in  c_{\mfu}(x(d)) $ with $Z^{'}(v_1)=x(d)v_2$ and $Z^{'}(v_{2})=0$, we have $Z^{'}=e_{2,n}$. Let $Z^{''}\in  c_{\mfu}(x(d))$ with $Z^{''}(v_1)=0$ and $Z^{''}(v_2)=x(d)^{n-2}v_{1}$, we have $Z^{''}=e_{1,n+1}$. In this case, we have
\[
 Z=\sum_{k=1}^{n-2}a_{k,1;1}x(d)^{k}+a_{1,2;1}e_{2,n}+a_{n-2,1;2}e_{1,n+1}.
\]
\end{case}

\begin{case}
  $\Delta(\mfm)=\{ \alpha_1, \alpha_s \}$  with $2<s<n$.  Let $x\in c_{\mfu}(x(d))$ with $x=e_{a,b}+x^{'}$ where $l(a)\neq l(b)$. If $a^{-}=b^{+}=\emptyset$, then $e_{a,b}=e_{1,s+1}$ or $e_{a,b}=e_{2,n+1}$, and $x^{'}\in c_{u}(x(d))$ by Lemma \ref{sp}.  Further, we have $e_{1,s+1}(v_{1})=0$ and $e_{1,s+1}(v_{2})=x(d)^{n-2}v_{1}$. By the same token, $e_{2,n+1}(v_{1})=x(d)v_{2}$ and $e_{2,n+1}(v_{2})=0$. If $a^{-}\neq \emptyset $ or $b^{+}\neq \emptyset$, 
then $e_{a,b}=e_{2,n}$ and $x-e_{2,n}-e_{s+1,n+1}\in c_{u}(x(d))$ by Lemma \ref{sp}. Moreover, we have 
\[
(e_{2,n}+e_{s+1,n+1})(v_{1})=v_{2}, \;\;\; (e_{2,n}+e_{s+1,n+1})(v_2)=0.
\]
\end{case}
  Since $e_{a,b}\neq e_{3,s+1}$, which implies that
$a_{n-3,1;2}=0$. In this case, we have
\[
 Z=\sum_{k=1}^{n-2}a_{k,1;1}x(d)^{k}+a_{n-2,1;2}e_{1,s+1}+a_{1,2;1}e_{2,n+1}
+a_{0,2;1}(e_{2,n}+e_{s+1,n+1}).
\]

\begin{case}
 $\Delta(\mfm)=\{ \alpha_r, \alpha_n \}$ with $1<r<n-1$.  Then $a_{0,2;1}=0$.  Let $x\in c_{\mfu}(x(d))$ with $x=e_{a,b}+x^{'}$. 
If $a^{-}=b^{+}=\emptyset$, then $e_{a,b}=e_{1,n+1}$ or $e_{a,b}=e_{r+1,n}$, and $x^{'}\in c_{u}(x(d))$ by Lemma \ref{sp}. 
Further, we have $e_{1,n+1}(v_{1})=0$ and $e_{1,n+1}(v_{2})=x(d)^{n-2}v_{1}$. By the same token, $e_{r+1,n}(v_{1})=x(d)v_{2}$ and $e_{r+1,n}(v_{2})=0$. 
If $a^{-}\neq \emptyset $ or $b^{+}\neq \emptyset$, 
then $e_{a,b}=e_{2,n+1}$ and $x-e_{2,n+1}-e_{1,r+1}\in c_{\mfu}(x(d))$ by Lemma \ref{sp}. Moreover, we have 
\[
(e_{2,n+1}+e_{1,r+1})(v_{1})=0 , \;\;\; (e_{2,n+1}+e_{1,r+1})(v_2)=x(d)^{n-3}v_{1}.
\]
In this case, we have
\[
 Z=\sum_{k=1}^{n-2}a_{k,1;1}x(d)^{k}+ a_{n-2,1;2}e_{1,n+1}+ a_{1,2;1}e_{r+1,n}+a_{n-3,1;2}(e_{2,n+1}+e_{1,r+1}).
\]
\end{case}

\begin{case}
 $\Delta(\mfm)=\{ \alpha_r, \alpha_s \}$ with $1<r<s-1<n-1$.
 Let $x\in c_{\mfu}(x(d))$ with $x=e_{a,b}+x^{'}$. 
If $a^{-}=b^{+}=\emptyset$, then $e_{a,b}=e_{1,s+1}$ or $e_{a,b}=e_{r+1,n+1}$, and $x^{'}\in c_{u}(x(d))$ by Lemma \ref{sp}. 
Further, we have $e_{1,s+1}(v_{1})=0$ and $e_{1,s+1}(v_{2})=x(d)^{n-2}v_{1}$. By the same token, $e_{r+1,n+1}(v_{1})=x(d)v_{2}$ and $e_{r+1,n+1}(v_{2})=0$. 
If $a^{-}\neq \emptyset $ or $b^{+}\neq \emptyset$, 
then $e_{a,b}=e_{1,r+1}$ and $x-e_{1,r+1}-e_{2,s+1}\in c_{\mfu}(x(d))$ or $e_{a,b}=e_{r+1,n}$ and $x-e_{r+1,n}-e_{s+1,n+1}\in c_{\mfu}(x(d))$ by Lemma \ref{sp}.    
Moreover, we have 
\[
(e_{1,r+1}+e_{2,s+1})(v_{1})=0 , \;\;\; (e_{1,r+1}+e_{2,s+1})(v_2)=x(d)^{n-3}v_{1}
\]
and
\[
(e_{r+1,n}+e_{s+1,n+1})(v_{1})=v_{2},   \;\;\; (e_{r+1,n}+e_{s+1,n+1})(v_{2})=0 
\]
In this case, we have
\begin{align*}
  Z=&\sum_{k=1}^{n-2}a_{k,1;1}x(d)^{k}+ a_{n-2,1;2}e_{1,s+1}+ a_{1,2;1} e_{r+1,n+1}\\
   &+ a_{n-3,1;2}(e_{1,r+1}+e_{2,s+1})+ a_{0,2;1}(e_{r+1,n}+e_{s+1,n+1}).
\end{align*}
\end{case}
\end{proof}

\begin{Cor}{\label{sum}}
Let $x(d)$ be a Richardson element determined by $|\Delta(\mfm)|\leq 2$ and $Z(c_{\mfu}(x(d)))$ be the center of its centralizer  $c_{\mfu}(x(d))$.  Then 
\begin{itemize}
\item[(1)] $Z(c_{\mfu}(x(d)))=c_{\mfu}(x(d))$, provided that $\Delta(\mfm)$ is one of the following cases
\[
\{ \alpha_{1}\}, \{\alpha_{n}\}, \{\alpha_{1}, \alpha_{2}\}, \{\alpha_{1}, \alpha_{n}\},
\{\alpha_{n-1}, \alpha_{n}\}.
\]
\item[(2)] $Z(c_{\mfu}(x(d)))$ is spanned by the power $x(d)^{k}$ of $x(d)$, provided that $\Delta(\mfm)$ is one of the following cases
\[
\{ \alpha_s\}, \{ \alpha_{r}, \alpha_{r+1}\}(1<r<n-1), \{ \alpha_{r},\alpha_{s} \}(1<r<s-1<n-1).
\]
\item[(3)] $Z(c_{\mfu}(x(d)))$ is spanned by $e_{2,n+1}$ together with $x(d)^{k}(1\leq k\leq n-2)$, provided that $\Delta(\mfm)=\{ \alpha_{1}, \alpha_{s} \}$ and
$2<s<n$.
\item[(4)]  $Z(c_{\mfu}(x(d)))$ is spanned by  $e_{1,n+1}$ together with $x(d)^{k}(1\leq k\leq n-2)$, provided that $\Delta(\mfm)=\{ \alpha_{r}, \alpha_{n} \}$ and $1<r<n-1$.
\end{itemize}
Moreover, the centralizer $c_{\mfu}(x(d))$ of $x(d)$ can be characterized by the following Table \ref{Tab1}
\def\arraystretch{2}
\begin{table}[ht]{\label{Tab1}}
\centering
\begin{tabular}[b]{|c|c|c|c|}
\hline
$\Delta(\mfm)$  &   restrictions   &  $c_{\mfu}(x(d))$ & dimension \\
\hline \hline
$\alpha_{1}$    &   none   &  abelian & $n$ \\
\hline
$\alpha_{s}$    &   $1<s<n$   &  not abelian  & $n+1$\\
\hline
$\alpha_{n}$    &   none   &  abelian & $n$ \\
\hline
$\alpha_{1},\alpha_{2}$     & none   & abelian & $n$  \\
\hline
$\alpha_{n-1},\alpha_{n}$   &   none   & abelian & $n$  \\
\hline
$\alpha_{r},\alpha_{r+1}$   &   $1<r<n-1$   & not abelian & $n+2$ \\
\hline
$\alpha_{1},\alpha_{n}$    &   none   & abelian & $n$  \\
\hline
$\alpha_{r},\alpha_{n}$          &   $1<r<n-1$   & not abelian  & $n+1$\\
\hline
$\alpha_{1},\alpha_{s}$          &   $2<s<n$   & not abelian & $n+1$  \\
\hline
$\alpha_{r},\alpha_{s}$          &   $1<r<s-1<n-1$   & not abelian & $n+2$  \\
\hline
\end{tabular}
\caption{Summarization for  $c_{\mfu}(x(d))$ }
\end{table}
\def\arraystretch{1}
\end{Cor}

\section{Saturation rank for nilpotent radical  }
Keep the notations as above, that is $G=\SL_{n+1}(K)$, $P=P(d)$ 
the standard parabolic subgroup of $G$ given by the set $\Delta({\mfm})$, $\mfp(d)=\mrlie(P(d))$ the Lie algebra of $P(d)$ and 
$\mfu$ the nilpotent ideal of $\mfp(d)$.
For any arbitrary element $x$ in $V(\mfu)$, we define the set
\[ 
\mbE(r,\mfu)_{x}:= \{\mfe \in \mbE(r,\mfu) \mid x\in \mfe \}
\]
as a subset of $\mbE(r,\mfu)$ consisting of elements that contain $x$.
Since $\mbE(1,\mfu)_{x}\neq \emptyset$, the number
\[ 
\rk(\mfu)_{x}:=\mmax\{r\in \mbN \mid \mbE(r,\mfu)_{x} \neq \emptyset\},
\]
is called the \emph{local saturation rank} of $x$. The first step towards 
the determination of the saturation rank of $\mfu$ is (see Sect. 3.1 in \cite{Pan}):

\begin{Lem}{\label{FD}}
Let $\rk_{min}(\mfu)=\min\{ \rk(\mfu)_{x} \mid x\in V(\mfu)\}$. Then 
$\srk(\mfu)=\rk_{min}(\mfu)$.
\end{Lem}

We consider the set 
\[
\mcO_{rmin}(\mfu):=\{ x\in V(\mfu) \mid \rk(\mfu)_{x}=\rk_{min}(\mfu) \}, 
\]
which is an open subset of $V(\mfu)$ (See  Sect. 3.1 in \cite{Pan} ).
Lemma \ref{FD} does not give a complete determination of the saturation rank of 
$\mfu$ because it does not say what the possible elements of $\mcO_{rmin}(\mfu)$ are. Recall that the nilpotent ideal $\mfu$ is the union of its intersection with the nilpotent orbits in $\mfg$. The finiteness of the number of nilpotent orbits ensures that there is a unique orbit $\mcO$,
such that $\mcO\cap \mfu $ is a $P(d)$-orbit and open dense in $\mfu$.
We call $\mcO$ the Richardson orbit corresponding to $P(d)$.

\begin{Lem}{\label{SD}}
If $V(\mfu)=\mfu$, then the saturation rank of $\mfu$ is determined by the local saturation rank of elements in $\mcO\cap \mfu$;
that is $\srk(\mfu)=\rk(\mfu)_{e}, \forall e\in \mcO \cap \mfu$.
\end{Lem}
\begin{proof}
For any $e\in \mcO\cap \mfu$,  the orbit
$\mcO \cap \mfu=P(d)\cdot e$ forms an open subset of $\mfu$. 
The fact that $\mcO_{rmin}$ is open in $V(\mfu)$ implies that it is also open in $\mfu$ since $V(\mfu)=\mfu$. It is important to note that
$\mfu$ is irreducible, and therefore, the intersection $P(d)\cdot e \cap \mcO_{rmin}$ is indeed non-empty. 
It is observed that the adjoint action of $P(d)$ on $\mfu$ remains within
$\mfu$, thereby implying that the local saturation rank $\rk(\mfu)_{e}$ of $e$ is equal to that of $p\cdot e$ for any $p\in P(d)$. 
As a result, $P(d)\cdot e\subset \mcO_{rmin}$, and 
$\srk(\mfu)=\rk(\mfu)_{e}$ for any $e\in \mcO\cap \mfu$.
\end{proof}

\begin{Lem}{\label{ineq}}
Let $x(d)$ be a Richardson element and $c_{\mfu}(x(d))$ be its centralizer in $\mfu$. Assume that $V(\mfu)=\mfu$, then
\begin{itemize}
\item[(1)] If $c_{\mfu}(x(d))$ is abelian, then 
$\srk(\mfu)=\dim c_{\mfu}(x(d))$.
\item[(2)] If $c_{\mfu}(x(d))$ is not abelian, then 
$\dim Z(c_{\mfu}(x(d))) \leq \srk(\mfu)< \dim c_{\mfu}(x(d))$.
\end{itemize}
\end{Lem}
\begin{proof}
It is observed that any $\mfe\in \mbE(\rk(\mfu)_{x(d)},\mfu)_{x(d)}$ is contained in $c_{\mfu}(x(d))$ and contains the center $Z(c_{\mfu}(x(d)))$.Thus, we have 
\[
\dim Z(c_{\mfu}(x(d)))\leq \rk(\mfu)_{x(d)} \leq \dim c_{\mfu}(x(d)).
\]
We immediately deduce the results in viewing of $\srk(\mfu)=\rk(\mfu)_{x(d)}$ by Lemma \ref{SD} since $x(d)$ belongs to $\mcO\cap \mfu$.
\end{proof}

We say a standard parabolic subgroup $P(d)$ of $G$ is restricted provided that $\mfu\subseteq V(\mfg)$, or, equivalently, that 
$\mcO \subseteq V(\mfg)$(cf. \cite{CLNP}).

\setcounter{case}{0}
\begin{Thm}{\label{s-r}}
Let $\mfp(d)$ be a parabolic subalgebra of $\mfsl_{n+1}(K)$ 
with $|\Delta(\mfm)|\leq 2$, and $\mfu$ be the nilpotent ideal of $\mfp(d)$. If $p\geq n+1$ or $P(d)$ is restricted, then the following statements hold:
\begin{itemize}
\item[(1)] The saturation rank of $\mfu$ is $n$.
\item[(2)]  Any maximal elementary subalgebra associated with a Richardson element $x(d)$ is either unique or parametrized by points of 
$\mathbb{P}^{1}$ or $\mathbb{P}^{1}\times \mathbb{P}^{1}$.
\item[(3)] The variety $\mbE(n,\mfu)_{x(d)}$ has a dimension of at most 2.
\end{itemize}
Moreover, we give the characterization of the variety $\mbE(n,\mfu)_{x(d)}$
in Table \ref{Tab2}.
\def\arraystretch{2}
\begin{table}[ht]{\label{Tab2}}
\centering
\begin{tabular}[b]{|c|c|c|c|}
\hline
$\Delta(\mfm)$  &   restrictions   &  $\mbE(n,\mfu)_{x(d)}$ & dimension \\
\hline \hline
$\alpha_{1}$    &   none   &  singleton & $0$ \\
\hline
$\alpha_{s}$    &   $1<s<n$   &  $\mathbb{P}^{1}$  & $1$\\
\hline
$\alpha_{n}$    &   none   &  singleton  & $0$ \\
\hline
$\alpha_{1},\alpha_{2}$     & none   & singleton & $0$  \\
\hline
$\alpha_{n-1},\alpha_{n}$   &   none   & singleton & $0$  \\
\hline
$\alpha_{r},\alpha_{r+1}$   &   $1<r<n-1$   & $\mathbb{P}^{1}\times \mathbb{P}^{1}$ & $2$ \\
\hline
$\alpha_{1},\alpha_{n}$    &   none   & singleton & $0$  \\
\hline
$\alpha_{r},\alpha_{n}$  &   $1<r<n-1$   & $\mathbb{P}^{1}$  & $1$\\
\hline
$\alpha_{1},\alpha_{s}$   &   $2<s<n$   & $\mathbb{P}^{1}$ & $1$  \\
\hline
$\alpha_{r},\alpha_{s}$   &   $1<r<s-1<n-1$   & $\mathbb{P}^{1}\cup \mathbb{P}^{1}$ & $1$  \\
\hline
\end{tabular}
\caption{Summarization for  $\mbE(n,\mfu)_{x(d)}$ }
\end{table}
\def\arraystretch{1}
\end{Thm}
\begin{proof}
In light of Corollary \ref{sum}, there are five types of $\Delta(\mfm)$
for which the centralizer $c_{\mfu}(x(d))$ is abelian and has dimension $n$. Consequently, the maximal elementary subalgebra containing $x(d)$ is $c_{\mfu}(x(d))$ itself, and by Lemma \ref{ineq}, $\srk(\mfu)=n$ when
$\Delta(\mfm)$ is as follows: 
\[
\{\alpha_{1}\}, \{\alpha_{n}\}, \{\alpha_{1}, \alpha_{2}\},  \{\alpha_{1}, \alpha_{n}\}, \{\alpha_{n-1}, \alpha_{n}\}.
\]
For the remaining cases, we will examine each one individually.
\begin{case}
$\Delta(\mfm)=\{ \alpha_{s} \}$ with $1<s<n$. The maximal elementary subalgebras containing $x(d)$ are
\[
\bigoplus_{i=1}^{n-1}Kx(d)^{i}\oplus K(ae_{1,s+1}+be_{s+1,n+1}) 
\]
which are  parametrized by points  $(a:b)\in  \mathbb{P}^{1}$.  
Consequently, the saturation rank is $n$, and the variety 
$\mbE(n, \mfu)_{x(d)}$ is irreducible with dimension 1.
\end{case}
\begin{case}
$\Delta(\mfm)=\{\alpha_{r},\alpha_{r+1}\}$   with  $1<r<n-1$.  The maximal elementary subalgebras containing $x(d)$ are
\[
\bigoplus_{i=1}^{n-2}Kx(d)^{i}\oplus K(ae_{1,r+1}+be_{r+1,n+1})\oplus K(a^{'}e_{1,r+2}+b^{'}e_{r+2,n+1})
\]
\end{case}
where $(a:b),(a^{'}:b^{'})\in \mathbb{P}^{1}$ and  $\mbE(n,\mfu)_{x(d)}$ is isomorphic to $\mathbb{P}^{1}\times \mathbb{P}^{1}$. So,
the saturation rank   is $n$ and the variety $\mbE(n,\mfu)_{x(d)}$ is irreducible with dimension 2.
\begin{case}
$\Delta(\mfm)=\{\alpha_{r},\alpha_{n}\}$  with  $1<r<n-1$. The maximal elementary subalgebras containing $x(d)$ are 
\[
\bigoplus_{i=1}^{n-2}Kx(d)^{i}\oplus K(a(e_{1,r+1}+e_{2,n+1})+be_{r+1,n})\oplus Ke_{1,n+1}
\]
which are  parametrized by points  $(a:b)\in  \mathbb{P}^{1}$.  Hence, the saturation rank   is $n$ and the variety $\mbE(n,\mfu)_{x(d)}$ is irreducible with dimension 1.
\end{case}
\begin{case}
$\Delta(\mfm)=\{\alpha_{1},\alpha_{s}\}$ with $2<s<n$. The maximal elementary subalgebras containing $x(d)$ are
\[
\bigoplus_{i=1}^{n-2}Kx(d)^{i}\oplus K(a(e_{2,n}+e_{s+1,n+1})+be_{1,s+1}) \oplus Ke_{2,n+1}
\]
which are  parametrized by points  $(a:b)\in  \mathbb{P}^{1}$.  Hence, the saturation rank is $n$  and  the variety $\mbE(n,\mfu)_{x(d)}$ is irreducible with dimension 1.
\end{case}
\begin{case}
$\Delta(\mfm)=\{\alpha_{r}, \alpha_{s}\}$  with  $1<r<s-1<n-1$.  
There exist two types of $n$-dimensional maximal elementary subalgebras  
\[
\bigoplus_{i=1}^{n-2}Kx(d)^{i}\oplus K(a(e_{1,r+1}+e_{2,s+1})+be_{r+1,n+1}) \oplus Ke_{1,s+1},
\]
\[
\bigoplus_{i=1}^{n-2}Kx(d)^{i}\oplus K(a^{'}(e_{r+1,n}+e_{s+1,n+1})+b^{'}e_{1,s+1}) \oplus Ke_{r+1,n+1}
\]
containing $x(d)$, both parametrized by points in  $\mathbb{P}^{1}$. 
Therefore, the saturation rank is $n$,  and the variety $\mbE(n,\mfu)_{x(d)}$ can be viewed as the union of two copies of $\mathbb{P}^{1}$, denoted as 
$\mathbb{P}^{1}\cup \mathbb{P}^{1}$. 
\end{case}
\end{proof}

\begin{remark}
We give two remarks concerning Theorem \ref{s-r}.
\begin{itemize}
\item[1.] As illustrated in Remark \ref{r1}, for $\mfp(d)\subseteq\mfsl_{9}(K)$ and $d=(3,2,3,1)$, the Richardson element is
$x(d)=e_{1,4}+e_{4,6}+e_{6,9}+e_{2,5}+e_{5,7}+e_{3,8}$ and there exists a maximal elementary subalgebra
\[
\bigoplus_{i=1}^{3}Kx(d)^{i}\oplus K(e_{1,5}+e_{4,7})\oplus
Ke_{1,7}\oplus Ke_{1,8}\oplus Ke_{2,8}\oplus Ke_{2,9}\oplus Ke_{3,7}\oplus Ke_{3,9}
\]
with dimesion $10$. Thereby, Theorem \ref{s-r}(1) may not hold for 
$|\Delta(\mfm)| > 2$.
\item[2.]Our fundamental application of  saturation rank determined in Theorem \ref{s-r} is given by the indecomposability of Carlson modules $L_{\zeta}$. 
Let $(P_{n},d_{n})_{n\geq 0}$ be a  minimal projective resolution of the trivial $U_{0}(\mfu)$-module $K$. Then 
\[
Hom(\Omega^{n}(K).K) \rightarrow  H^{n}(\mfu,K), \;\;\; \hat{\zeta} \mapsto  [\hat{\zeta} \circ d_{n}],
\]
is an isomorphism. If $\zeta:=[\hat{\zeta} \circ d_{n}]\neq 0$, then the Carlson module is defined as $L_{\zeta}:=\Ker\hat{\zeta}\subseteq \Omega^{n}(K)$.
By virtue of Theorem 6.4.4 in \cite{Far}, under the conditions given in Theorem \ref{s-r},  if $n\geq 2$ and $\zeta\neq 0 $ has odd degree, then $L_{\zeta}$ is indecomposable.
\end{itemize}
\end{remark}

\end{document}